\documentclass[12 pt]{amsart}

\usepackage{graphicx}
\usepackage{color}
\usepackage{amssymb, amsmath, amsthm}
\usepackage{mathptmx}
\usepackage{comment}

\usepackage{epsfig}
\usepackage{epstopdf}
\usepackage{graphicx}
\usepackage{pinlabel}
\usepackage{enumerate}

\theoremstyle{plain}

\newtheorem{theorem}{Theorem}
\newtheorem{lemma}{Lemma}
\newtheorem{corollary}{Corollary}

\newtheorem{question}{Question}

\newtheorem{case}{Case}

\newcommand{\R}{\mathbb R} %REALS

\newcommand{\bi}{\begin{itemize}}
\newcommand{\ei}{\end{itemize}}
\newcommand{\be}{\begin{enumerate}}
\newcommand{\ee}{\end{enumerate}}

\newcommand{\n}{\beta}

\newcommand{\emp}{\emptyset}
\newcommand{\X}{\times}

\newcommand{\A}{\alpha}
\newcommand{\pd}{\partial}

\numberwithin{definition}{section}
\numberwithin{example}{section}
\numberwithin{lemma}{section}
\numberwithin{theorem}{section}
\numberwithin{corollary}{section}

\begin{document}
\title{Knots with compressible thin levels}
\author[R. Blair]{Ryan Blair}
\author[A. Zupan]{Alexander Zupan}

\thanks{The second author is supported by the National Science Foundation under Award No. DMS-1203988.}

\maketitle

\begin{abstract}
We produce embeddings of knots in thin position that admit compressible thin levels. We also find the bridge number of tangle sums where each tangle is high distance.
\end{abstract}

\section{Introduction}

Thin position has contributed to many advances in the study of knots and 3-manifolds. Gabai defined thin position for knots in $S^3$ and employed the notion in his proof of Property R~\cite{Ga}. Subsequently, thin position was utilized in the solution to the knot complement problem~\cite{gl}, the recognition problem for $S^3$~\cite{R,AT}, and the leveling of unknotting tunnels~\cite{GST}.

The usefulness of thin position for knots in $S^3$ is grounded in its connections to the topology of the knot exterior. In particular, Thompson showed that a knot in thin position that is not bridge position contains an essential planar surface in its exterior~\cite{thomps}. Wu strengthened this result to show that every thinnest thin level of a knot in thin position is essential in the knot exterior \cite{wu}. Additional results have furthered our understanding of essential surfaces and thin position \cite{HK,tomova1}. These results motivate the natural question,

\begin{question}
Given a knot in thin position, are all thin levels essential in the knot exterior?
\end{question}

This question was originally considered by Thompson during her work on the recognition problem for $S^3$. In the context of that problem as well as many others, the possibility of a knot or graph in thin position admitting a compressible thin level presents a significant technical challenge. It has long been believed that these challenges are endemic and that there exist knots in thin position that admit a compressible thin level. In fact, candidate examples are well known in the community. See Figure \ref{Fig:thinK}. However, demonstrating that one of these candidate embeddings is in thin position has proved to be a difficult problem. In this paper, we use recent advancements in the study of the distance of bridge surfaces \cite{bty,johntom} to prove the following:

\begin{theorem}\label{main}
There are infinitely many knots $K$ that admit a thin position with a compressible thin level.
\end{theorem}

Our main result supports the emerging theme that the fundamental properties of thin position are quite subtle. Unlike the closely related invariant of bridge number, width (an integer invariant derived from thin position) is not additive with respect to connected sum~\cite{blairtom}. Additionally, thin position of a knot may not minimize bridge number~\cite{blairtom}. Finally, there exist embeddings of the unknot that cannot be isotoped to the standard unknot through embeddings of non-increasing width~\cite{Z}.

Our work also reinforces the notion that knots and 3-manifolds composed of sufficiently complicated pieces have predictable topology; the proof of Theorem \ref{main} demonstrates that the obvious embeddings of the candidate knots $K$ coincide with their thin positions.  In terms of 3-manifolds, this concept is apparent in work of Kobayashi-Qui \cite{KobQui}, which demonstrates that if sufficiently complicated compact 3-manifolds are glued together along a common boundary component, then the Heegaard genus of the resulting 3-manifold is as expected.

In the setting of knots in $S^3$, the natural analogue of 3-manifold amalgamation is the gluing together of two collections of arcs each contained in a 3-ball, an operation known as tangle summation.  In general, it is difficult to predict the bridge number of a tangle sum; for instance, the tangle sum of two nontrivial tangles can give rise to the unknot.  In \cite{Blair1}, the first author studies the degeneration of bridge number for tangle sums of 2-strand tangles, and in \cite{Blair2}, he gives a lower bound for the bridge number of a tangle sum with some mild restrictions on the tangles being glued together.  However, as with 3-manifolds, if we restrict our tangles by requiring that they have sufficiently complicated bridge surfaces, the picture becomes much clearer.

Using the techniques mentioned above, we prove the complete analogue of the main theorem in \cite{KobQui} for the bridge number of knots in $S^3$ (see Section \ref{bridge-surface} for relevant definitions):

\begin{theorem}\label{main2}
Suppose that $(B_1,\tau_1)$ and $(B_2,\tau_2)$ are $n$-strand tangles with $n>1$, and let $K$ be a knot in $S^3$ that is the tangle sum of $(B_1,\tau_1)$ and $(B_2,\tau_2)$.  In addition, let $\Sigma_i$ be a $\n_i$-bridge sphere for $(B_i,\tau_i)$, where $\n_i > n$.  If $d(\Sigma_i) > 2(\n_1 + \n_2 - n)$ for $i=1,2$, then
\[ b(K) = \n_1 + \n_2 - n.\]
\end{theorem}

Our paper is structured as follows. In Section \ref{prelim}, we introduce the notions of width and bridge number and we present relevant background. In Section \ref{bridge-surface}, we define distance of a bridge surface and give results connecting distance to the existence of essential surfaces and alternate bridge surfaces. In Section \ref{tangle-sum}, we prove Theorem \ref{main2}. In Section \ref{high-distance}, we give an embedding of a knot with a compressible thin level, while in Sections \ref{Sec:DefinitionK}, \ref{Sec:thinposition}, and \ref{final-case} we show that this embedding is a thin position for the knot. Our proof relies on recent advances in the study of high distance bridge surfaces~\cite{johntom} and the construction of our examples is inspired by the examples of strict subadditivity of width presented in~\cite{blairtom}.

\section{Preliminaries}\label{prelim}
A \emph{knot} $K$ is an isotopy class of embedded simple closed curves in the 3-sphere.  For the remainder of the paper, we fix a Morse function $h:S^3 \rightarrow \R$ such that $h$ has exactly two critical points, one of index zero and one of index three.  Given an embedded simple closed curve $k$ in $S^3$, we may perturb $k$ (if necessary) so that $h|_k$ is Morse.  Let $c_0,\dots,c_n$ denote the critical values of $h|_k$ and choose regular values $r_1,\dots,r_n$ satisfying $c_{i-1} < r_i < c_i$.  We say that $h^{-1}(r_i)$ is a \emph{level sphere} having \emph{width} $w(h^{-1}(r_i)) = |k \cap h^{-1}(r_i)|$.  The \emph{width} $w(k)$ and \emph{bridge number} $b(k)$ of $k$ are defined to be
\[ w(k) = \sum w(h^{-1}(r_i)) \, \, \text{ and } \, \, b(k) = \frac{n+1}{2},\]
respectively.  To obtain invariants of the knot $K$, we minimize these two quantities over all possible embeddings of $K$.  In other words, the \emph{width} $w(K)$ and \emph{bridge number} $b(K)$ of $K$ are defined by
\[ w(K) = \min_{k \sim K} w(k) \,\, \text{ and } \, \, b(K) = \min_{k \sim K} b(k).\]
We say that $k$ is a \emph{thin position} of $K$ if $w(k) = w(K)$.   More generally, we say $k$ is a \emph{bridge position} if all maxima of $h|_k$ occur above all minima of $h|_k$, and a bridge position $k$ is a \emph{minimal bridge position} if $b(k) = b(K)$. Equivalently, we could define these concepts by fixing an embedding $k$ of $K$ and considering isotopy classes of Morse functions $h$ on $K$.

Let $k$ and $r_1,\dots,r_n$ be as above.  In \cite{ss}, Scharlemann and Schultens present an alternate formula for calculating width using level spheres which intersect $k$ maximally or minimally.  For $1 < i < n$, we call $h^{-1}(r_i)$ a \emph{thick level} if $c_{i-1}$ is a minimum and $c_i$ is a maximum, or a \emph{thin level} if $c_{i-1}$ is a maximum and $c_i$ is a minimum.  It is a straightforward exercise to see that if $k$ has $m$ thick levels, then it has $m-1$ thin levels.  We let $a_1,\dots,a_m$ (resp. $b_1,\dots,b_{m-1})$ denote the widths of the thick levels (resp. thin levels), where both sets of numbers are naturally ordered by the height function $h$.  Hence, each $k$ gives rise to a tuple of even integers $(a_1,b_1,a_2,\dots,b_{m-1},a_m)$ which we call the \emph{thin-thick} tuple corresponding to $k$.  By \cite{ss},
\begin{equation}\label{alt}
w(k) = \frac{1}{2} \left(\sum a_i^2 - \sum b_i^2 \right).
\end{equation}
In particular, if we select one of the $a_i$'s, we have
\begin{equation}\label{ineq1}
w(k) \geq \frac{a_i^2}{2}.
\end{equation}

As demonstrated by Thompson \cite{thomps} and Wu \cite{wu}, thin position is related to essential surfaces embedded in the exterior $E(K) = S^3 \setminus \eta(K)$ of $K$, where $\eta( \cdot )$ denotes an open regular neighborhood.  A compact surface $S$ with nonempty boundary properly embedded in $E(K)$ is said to be \emph{meridional} if each curve of $\pd S$ bounds a meridian disk of the solid torus $\overline{\eta(K)}$.  As such, meridional surfaces are in one-to-one correspondence with closed surfaces embedded in $S^3$ and intersecting $K$ transversely.  We often blur the distinction between such surfaces; however, for a closed surface $S$ in $S^3$ intersecting $K$ transversely, we will use $S_K$ to denote $S \cap E(K)$ where appropriate.  In an abuse of terminology, we also occasionally refer to $S_K$ as a \emph{punctured} surface despite the fact that it is compact.  When we say that two closed surfaces $S$ and $T$ in $S^3$ which intersect $K$ transversely are isotopic, we will mean that the surfaces are isotopic relative to $K$ (equivalently, $S_K$ is isotopic to $T_K$ in $E(K)$) unless otherwise specified.

Suppose now that $S$ is a properly embedded meridional or closed surface in $E(K)$.  A \emph{compressing disk} $D$ for $S$ is an embedded disk such that $S \cap D = \pd D$ and $\pd D$ is an essential curve in $S$.  If there is an compressing disk for $S$, we say $S$ is \emph{compressible}; otherwise, $S$ is \emph{incompressible}.  If $S$ is isotopic into $\pd E(K)$, we say $S$ is \emph{$\pd$-parallel}, and in the case that $S$ is incompressible and not $\pd$-parallel, we call $S$ \emph{essential}.

There are two other classes of disks we will use, and to define these we consider $S$ as an embedded surface in $S^3$ intersecting $K$ transversely.  A \emph{bridge disk} is an embedded disk $\Delta$ such that $\pd \Delta$ is the endpoint union of arcs $\A$ and $\n$, where $\Delta \cap S = \A$ and $\n \subset K$.  A \emph{cut disk} $C$ is an embedded disk such that $C \cap S = \pd C$, $\pd C$ is essential in $S_K$, and $C \cap K$ is a single point in the interior of $C$.  If there is a compressing or cut disk (a \emph{c-disk}) for $S$, we way $S$ is \emph{c-compressible}; otherwise $S$ is \emph{c-incompressible}.  If $S$ is c-incompressible and not $\pd$-parallel, $S$ is \emph{c-essential}.

If a surface $S$ is c-compressible, then we may surger $S$ along a c-disk $D$ to get a new surface $S'$.  We note that this process has an inverse operation:  If $D$ is a compressing disk, we may recover $S$ from $S'$ by performing surgery on $S'$ along an arc $\A$ such that $\A \cap S' = \pd \A$; that is, viewing $\overline{\eta(\A)}$ as $\A \X \Delta$, where $\Delta$ is a small disk, we glue the annulus $\A \X \pd \Delta$ to $S' \setminus \eta(\pd\A)$.  If $D$ is a compressing disk, we recover $S$ using an arc $\A$ disjoint from $K$; when $D$ is a cut disk, we recover $S$ using an arc $\A$ contained in $K$.  In either case, we say that $S$ results from \emph{tubing} $S'$.

The relationship between thin position of a knot $K$ and essential meridional surfaces is made precise with the following theorem:

\begin{theorem}\cite{wu}\label{thinner}
Suppose $k$ is a thin position of $K$ which has a thin level.  Then, any thinnest thin level $h^{-1}(r_i)$ is an essential meridional surface in $E(K)$.
\end{theorem}

A knot $K$ in $S^3$ is \emph{prime} if its exterior contains no essential meridional annulus.  The above result has been strengthened by Tomova, and we will employ the following extension:

\begin{theorem}\cite{tomova1}\label{plustwo}
Suppose $K$ is a prime knot in thin position and $P$ is a thin sphere of minimal width.  If $P'$ is another thin sphere satisfying $w(P') = w(P) + 2$, then $P'$ is incompressible.
\end{theorem}

We note, however, that the preceding theorems do not imply that thin spheres are $c$-incompressible.  In the case of a c-compressible thin sphere, we will be able to use a theorem of Blair and Tomova.

\begin{theorem}\cite{blairtom}\label{thinnest}
Let $K$ be a prime knot in thin position and suppose $P$ is a c-compressible thin sphere, with a c-disk $D$.  Then there is a thin sphere $P'$ adjacent to $P$ in the direction of $D$ such that either $D \cap P' \neq \emp$ or $w(P') < w(P)$.
\end{theorem}

\section{Bridge surfaces and distance}\label{bridge-surface}

In this section, we will define bridge surfaces for knots and tangles and discuss important relationships between bridge surfaces, essential surfaces, and distances between disk sets in the curve complex.

A \emph{punctured 3-sphere} is the complement of a disjoint embedded collection of open 3-balls in $S^3$.  We construct our knots by gluing together arcs embedded in punctured 3-spheres, called tangles.  A \emph{tangle} $(B,\tau)$ consists of a 3-manifold $B$ homeomorphic to a punctured 3-sphere containing a properly embedded 1-manifold $\tau$ which has no closed components.  In the special case that $B$ is a 3-ball and $n = |\tau|$, we call $(B,\tau)$ an \emph{$n$-strand tangle}.  Every embedded 2-sphere $\Sigma$ in a punctured 3-sphere $B$ separates $B$ into two punctured 3-balls $B_1$ and $B_2$.  If $B_i$ has $n \geq 2$ boundary components, then there is a collection of $n-2$ properly embedded arcs $\Gamma_i$, called the \emph{spine} of $B_i$, such that $B_i \setminus \eta(\Gamma_i)$ is homeomorphic to $\Sigma \times I$. If $B_i$ has one boundary component, then $\Gamma_{i}$ is a single point in the interior of $B_{i}$.

Given $(B,\tau)$ such that $(B,\tau)$ is a tangle or $B=S^3$ and $\tau$ is a knot, we say an embedded 2-sphere $\Sigma$ in $B$ is a \emph{bridge sphere} if $\Sigma$ is transverse to $\tau$ and if for each arc $\alpha$ of $\tau \cap B_i$ with both endpoints on $\Sigma$ there is a bridge disk whose boundary contains $\alpha$ (such arcs called \emph{$\pd$-parallel}), and every other component of $\tau \cap B_i$ is an arc isotopic to an I-fiber of $B_i \setminus \eta(\Gamma_i)$, $i=1,2$ (such arcs are called \emph{vertical}).  Letting $\tau_i = \tau \cap B_i$, we say that $(B_i,\tau_i)$ is a \emph{trivial tangle} and we express the decomposition of $(B,\tau)$ by $\Sigma$ as a \emph{bridge splitting},
\[ (B,\tau) = (B_1,\tau_1) \cup_{\Sigma} (B_2,\tau_2).\]

Here we make several other definitions concerning tangles:  A tangle $(B,\tau)$ is \emph{irreducible} if every 2-sphere embedded in $B \setminus \eta(\tau)$ bounds a ball, and $(B,\tau)$ is \emph{prime} if every meridional annulus in $B \setminus \eta(\tau)$ is boundary parallel.

Given an embedding of a knot in $S^3$ that is Morse with respect to the height function $h$, every thick level $\Sigma$ is a bridge sphere for the portion of the knot that lies between the thin level immediately below $\Sigma$ and the thin level immediately above $\Sigma$. If $(B,\tau)$ is a knot in $S^3$ or a tangle in a 3-ball, we say a bridge sphere $\Sigma$ is an \emph{$n$-bridge sphere} if $(B_1,\tau_1)$ is an $n$-strand tangle. It is a simple exercise to verify that if $k$ is an embedding of a knot $K$ with one thick level $h^{-1}(r)$, then $h^{-1}(r)$ is an $n$-bridge sphere, where $n = w(h^{-1}(r))/2$.  As above, if $S$ is an embedded surface in $B$ transverse to $\tau$, we will use $S_{\tau}$ to denote $S \setminus \eta(\tau) \subset B \setminus \eta(\tau)$ when we wish to refer to a surface in the exterior of $\tau$ in $B$.

We will utilize ``sufficiently complicated" bridge surfaces, where complication is measured via distances between disk sets in the curve complex.  For a compact surface $S$, the (1-skeleton of the) curve complex $\mathcal{C}(S)$ is a graph whose vertices are isotopy classes of essential simple closed curves.  Two vertices are connected by an edge if their corresponding curves may be realized disjointly.  The vertex set of the curve complex $\mathcal{C}(S)$ has a natural metric constructed by assigning each edge length one and defining the distance between two vertices to be the length of the shortest path between them.

Given a bridge sphere $\Sigma$ splitting $(B,\tau)$ into $(B_1,\tau_1) \cup_{\Sigma} (B_2,\tau_2)$, we define the disk set $\mathcal{D}_i \subset \mathcal{C}(\Sigma_{\tau})$ to be those curves in $\Sigma_{\tau}$ which bound compressing disks in $B_i \setminus \eta(\tau_i)$, and the distance $d(\Sigma)$ of the splitting to be
\[ d(\Sigma) = \min\{ d(c_1,c_2): c_i \in \mathcal{D}_i\}.\]
The utility of this definition is made clear by the following two theorems of Johnson and Tomova (the first is a version of Theorem 4.4 from \cite{johntom} and the second is implicit in the proof of Theorem 4.2):

\begin{theorem}\cite{johntom}\label{tangdist1}
Suppose $(N,K)$ is a tangle or $N = S^3$ and $K$ is a knot, and let $(M,T)$ be a tangle inside of $(N,K)$ such that $K$ is transverse to $\pd M$ and $(\pd M)_K$ is incompressible in $(\overline{N\setminus M}) \setminus \eta(K)$, and $T = K \cap M$.  In addition, let $\Sigma$ be a bridge sphere for $(M,T)$ and let $\Sigma'$ be a bridge sphere for $(N,K)$.  Then one of the following holds:
{\bi
\item After isotopy, some number of c-compressions of $\Sigma_K' \cap M$ in $M$ yield a compressed surface $\Sigma''_K$ such that $\Sigma_K'' \cap M$ is parallel to $\Sigma_K$.
\item $d(\Sigma) \leq 2 - \chi(\Sigma'_K)$.
\item $\chi(\Sigma_K) \geq -3$.
\ei}
\end{theorem}

\begin{theorem}\cite{johntom}\label{tangdist2}
Suppose $(N,K)$ is a tangle or $N = S^3$ and $K$ is a knot, and let $(M,T)$ be a tangle inside of $(N,K)$ such that $K$ is transverse to $\pd M$ and $(\pd M)_K$ is incompressible in $(\overline{N\setminus M}) \setminus \eta(K)$, and $T = K \cap M$.  Let $\Sigma$ be a bridge sphere for $(M,T)$ and $F$ be a surface properly embedded in $N$ such that $F$ is meridional, incompressible and separating. Additionally, suppose $F \cap (\pd M)_K$ is essential in $(\pd M)_K$. Then one of the following holds:
{\bi
\item $d(\Sigma) \leq 2 - \chi(F_K)$.
\item $\chi(\Sigma_K) \geq -3$.
\item Each component of $F \cap (M \setminus \eta(T))$ is boundary parallel in $M \setminus \eta(T)$.
\ei}
\end{theorem}

As an immediate application of Theorem \ref{tangdist2}, we prove the next two lemmas for use later in the paper.
\begin{lemma}\label{prime}
Suppose $(B,\tau)$ is a tangle containing a bridge sphere $\Sigma$.  Then one of the following holds:
{\bi
\item $d(\Sigma) \leq 2$.
\item $\chi(\Sigma_{\tau}) \geq -3$.
\item $(B,\tau)$ is prime.
\ei}
\begin{proof}
Suppose that the second two conclusions do not hold, so that $(B,\tau)$ contains an essential meridional annulus $A$.  By Theorem \ref{tangdist2}, it follows that $d(\Sigma) \leq 2- \chi(A) \leq 2$.
\end{proof}
\end{lemma}

\begin{lemma}\label{essdist}
Suppose $K$ is a prime knot in $S^3$, and $(S^3,K)$ contains a tangle $(B,\tau)$ with bridge sphere $\Sigma$.  Suppose $F$ is a closed, separating surface in $S^3$ intersecting $K$ transversely such that $F_K$ is incompressible in $E(K)$. Additionally, suppose that $(\pd B)_{K}$ is incompressible in the knot complement. Then one of the following holds:
{\bi
\item $d(\Sigma) \leq 2 - \chi(F_K)$.
\item $\chi(\Sigma_K) \geq -3$.
\item After some number of cut compressions and an isotopy of $F$, we have $F \cap B = \emp$.  If, in addition, $F$ is isotopic into $B$, then $F_K$ is parallel to a component of $(\pd B)_K$.
\ei}
\begin{proof}
Since $(\pd B)_{K}$ and $F_K$ are incompressible meridional surfaces and $E(K)$ is irreducible, we can assume after isotopy that $(\pd B)_{K}\cap F_K$ is a collection of curves that do not bound disks in either surface. Further, since $E(K)$ is prime, we may remove curves of intersection of $(\pd B)_K \cap F_K$ which are boundary parallel in $(\pd B)_K$ by either cut compressions or isotopy of $F_K$, after which we may assume that $(\pd B)_K \cap F_K$ consists of curves essential in $(\pd B)_K$.

By Theorem \ref{tangdist2}, if the first two conclusions do not hold, then each component $R_i$ of $F \cap (B \setminus \eta(\tau))$ is parallel to a subsurface $R_i'$ of $S = \pd(B \setminus \eta(\tau))$.  Let $\{R_i'\}$ denote the collection of all such subsurfaces, and isotope $F$ so that the total number of components of intersections of surfaces $R_i' \in \{R_i'\}$ with $\pd \eta(\tau)$ is minimal up to cut compressions and isotopy of $F_K$.  Now, pick an index $i$ so that the corresponding $R_i'$ is innermost in $S$.  If there is an arc $t \in \tau$ such that $\pd \eta(t) \subset R_i'$, then a meridian curve $\gamma'$ of $\eta(t)$ lifts to curve $\gamma \subset R_i$.  In $B$, the curve $\gamma$ bounds a disk $C$ which intersects $\tau$ in exactly one point, so that either $C$ is a cut disk, or we could remove the intersection of $R_i'$ and $\pd \eta(t)$ via an isotopy.  In either case, we reduce the number of intersections of $\{R_i'\}$ with $\pd \eta(\tau)$.

We conclude that up to cut-compressions and isotopy of $F_K$, $R_i' \cap \eta(\tau) =\emp$ for all $R_i'$, so every component $R_i$ of the cut-compressed $F$ intersected with $B \setminus \eta(\tau)$ is parallel into $\pd B_{\tau}$.  Thus, we may push all such $R_i$ outside of $B$.  If, in addition, $F$ is isotopic into $B$, then $\{R'_i\}$ contains exactly one region $R_i'$, and $R_i'$ can be chosen so that $\pd R_i'$ avoids both $\text{int}(\pd(\eta(\tau)))$ and $\text{int}((\pd B)_K)$.  Thus, $R_i'$ is a component of $(\pd B)_K$, as desired.
\end{proof}
\end{lemma}

Lastly, using results of Blair-Tomova-Yoshizawa \cite{bty}, it is shown in \cite{blairtom} that high distance tangles exist:

\begin{theorem}
Given positive integers $b,D,n$ with $n<b$ and $b \geq 3$, there exists an $n$-strand tangle $(B,\tau)$ with a bridge sphere $\Sigma$ such that $|\tau \cap \Sigma|  = 2b$ and $d(\Sigma) \geq D$.
\end{theorem}

We remark that the existence of high distance tangles also follows from \cite{JM}, discussed in Section \ref{high-distance}.

\section{Tangle sum and bridge number}\label{tangle-sum}
In this section, we show that if two $n$-strand tangles $(B,\tau)$ and $(B',\tau')$ are glued together to form a knot $K$ in $S^3$ and both of the tangles have a bridge sphere of sufficiently high distance, then we may determine the bridge number of the resulting knot.  In this case, we say that $K$ is the \emph{tangle sum} of $(B,\tau)$ and $(B',\tau')$.  Fist, we make several definitions about the intersections of various c-disks for a bridge sphere for $K$.

Suppose $\Sigma$ is a bridge sphere for a knot $K$.  If there are c-disks $D$ and $D'$ on opposite sides of $\Sigma$ such that $D \cap D' = \emp$, we say that $K$ is \emph{c-weakly reducible}.  Otherwise, $K$ is \emph{c-strongly irreducible}.

\begin{lemma}\label{cthin}
Suppose $k$ is an embedding of $K$ such that $k$ has one thick level, a bridge sphere $\Sigma$.  If $\Sigma$ is c-weakly reducible.  Then we may exhibit an embedding $k'$ of $K$ such that $b(k') = b(k)$ and $w(k') < w(k)$.
\begin{proof}
Suppose $(S^3,K) = (B_1,\tau_1) \cup_{\Sigma} (B_2,\tau_2)$, and let $D_1$ and $D_2$ be c-disks in $B_1$ and $B_2$, respectively, such that $D_1 \cap D_2 = \emp$.  Since $\Sigma$ is a 2-sphere, there exist disks $C_1$ and $C_2$ in $\Sigma$ (these disks necessarily intersect $K$) such that $\pd C_i = \pd D_i$ and $C_1 \cap C_2 = \emp$.  Since $D_i$ is a c-disk, we may choose an arc $t_i \in \tau_i$ such that $\pd t_i \subset C_i$.  Let $\Delta_i$ be a bridge disk for $t_i$ that intersects $D_i$ minimally.  By a standard cut-and-paste argument, we may show that $\Delta_i \cap D_i = \emp$; hence $\Delta_i \cap \Sigma \subset C_i$.  It follows that $\Delta_1 \cap \Delta_2 =\emp$, and there is an isotopy of $k$ supported in $\eta(\Delta_1) \cup \eta(\Delta_2)$ which moves a minimum above a maximum, yielding $k'$ such that $b(k')=b(k)$ and $w(k') = w(k) - 4$.
\end{proof}
\end{lemma}

Now we connect the distance of tangles to c-weak reducibility. To do this, we need the following lemma due to the second author.

\begin{lemma}\cite{Z2}\label{nonesting}
Suppose $\Sigma$ is a $c$-strongly irreducible bridge sphere for a knot $K$ in $S^3$, where $(S^3,K)=(B_1,\tau_1)\cup_{\Sigma}(B_2,\tau)$. If $c$ is an essential curve in $\Sigma_K$ such that $c$ bounds a disk $D$ embedded in $E(K)$, where a collar of $c$ in $D$ is disjoint from $\Sigma_K$, then $c$ bounds a compressing disk for $\Sigma_K$ in $B_i \setminus \eta(\tau_i)$ for $i = 1$ or $2$.
\end{lemma}

Next, we synthesize Lemma \ref{nonesting} and Theorem \ref{tangdist1} into a single result.

\begin{lemma}\label{tangdist3}
Suppose that $K$ is a prime knot in $S^3$, $(S^3,K)$ contains a 3-ball $M$ with $K$ transverse to $\pd M$ and $(\pd M)_{K}$ is incompressible in $E(K)$.  Let $T = K \cap M$.  In addition, let $\Sigma$ be a bridge sphere for $(M,T)$ and let $\Sigma'$ be a bridge sphere for $(S^3,K)$.  Then one of the following holds:
{\bi
\item After isotopy, $\Sigma' \cap \pd M = \emp$.
\item $\Sigma'$ is c-weakly reducible.
\item $d(\Sigma) \leq 2 - \chi(\Sigma'_K)$.
\item $\chi(\Sigma_K) \geq -3$.
\ei}
\begin{proof}
By Theorem \ref{tangdist1}, if the third and fourth conclusions do not hold, then after isotopy, there exists a collection $\{D_i\}$ of c-compressing disks of $\Sigma_K' \cap M$ in $M$ such that compressing along $\{D_i\}$ yields a compressed surface $\Sigma''_K$ such that $\Sigma_K'' \cap M$ is parallel to $\Sigma_K$, and we may choose $\{D_i\}$ so that the collection $\{\pd D_i\}$ is pairwise disjoint.  Let $\{\gamma_i\}$ denote the collection of curves in $\Sigma_K'$ such that $\gamma_i = \pd D_i$.  If any $\gamma_i$ is inessential in $\Sigma_K'$, pick an innermost such $\gamma_i$.  The curve $\gamma_i$ bounds a disk or cut disk $D'$ contained in $\Sigma'_K$ and since $\gamma_i$ is essential in $\Sigma'_K \cap M$, we have $D' \cap \pd M \neq \emp$.  In this case, since $K$ is prime and $E(K)$ is irreducible, compressing along $D_i$ and discarding a trivial 2-sphere or $\pd$-parallel annulus yields the same surface as isotoping $D'$ onto $D_i$, and this isotopy yields a new surface $\Sigma^*_K$ which intersects $\pd M$ fewer times.

If every curve in $\{\gamma_i\}$ is inessential in $\Sigma'_K$, we can repeat the above process to show that after isotopy, $\Sigma'_K = \Sigma''_K$, and as such $\Sigma'_K$ may be made disjoint from $\pd M$.  Otherwise, some $\gamma_i$ is essential in $\Sigma'_K$. Let $D_j$ be the first c-disk in the sequence of c-compressions such that $\gamma_j$ is essential in $\Sigma'_K$. Hence $\gamma_j$ bounds a c-disk for $\Sigma'$. As $\Sigma''$ is a bridge sphere for $(M,T)$, we can find compressing disks $C$ and $C'$ for $\Sigma''_K$ which are regular neighborhoods of bridge disks $\Delta$ and $\Delta'$ on opposite sides of $\Sigma''$ such that $\Delta$ and $\Delta'$ meet in a point $p$ contained in $K$. Let $c$ and $c'$ be the boundary curves of $C$ and $C'$ in $\Sigma''_K$. Since $\gamma_j$ bounds a c-compressing disk taking $\Sigma'_K$ to $\Sigma''_K$ and $c$ and $c'$ are contained in $\Sigma''_K$, both $c$ and $c'$ are disjoint from $\gamma_j$.

To form a contradiction, we assume that $\Sigma'$ is c-strongly irreducible.  We may recover $\Sigma'_K$ by tubing $\Sigma''_K$ to a collection of meridional surfaces some number of times.  These tubes may be chosen to be disjoint from $c$ and $c'$; hence, we may view $c$ and $c'$ as essential curves in $\Sigma'_K$ that bound disks meeting the hypothesis of Lemma \ref{nonesting}. Since we have assumed that $\Sigma'$ is strongly irreducible, then, by Lemma \ref{nonesting}, both $c$ and $c'$ bound compressing disks $E$ and $E'$ for $\Sigma'_K$. If $E$ and $E'$ are on opposite sides of $\Sigma'_K$, then $D_j$ and one of $E$ and $E'$ are c-disks for $\Sigma'_K$ embedded on opposite sides of $\Sigma'_K$, a contradiction to our assumption. Hence, we can assume that $E$ and $E'$ are embedded on the same side of $\Sigma'$.

By construction, $c$ bounds a twice-punctured disk $F \subset \Sigma''_K$ and $c'$ bounds a twice-punctured  disk $F' \subset \Sigma''_K$, where $F \cap F'$ is a collection of unpunctured disks in $\Sigma''_K$ together with a single disk component $F''$ containing the single puncture $p$ named above.  Let $\alpha$ and $\alpha'$ denote the arcs in $c$ and $c'$ which cobound $F''$.

After the tubing operations, $\alpha \cup \alpha'$ bounds a punctured disk $G''$ in $\Sigma'_K$ which is the intersection of punctured disks $G$ and $G'$ bounded by $c$ and $c'$ in $\Sigma'_K$.  Moreover, by construction every puncture contained in $G \cap G'$ is contained in $G''$ since we can assume that all components of $F \cap F'$ other than $F''$ are disjoint from the tubes in the tubing operations.  It is possible that in $\Sigma'_K$, $G''$ contains more than a single puncture; however, $|K \cap G''| \equiv |K \cap F''| \, (\text{mod } 2)$, since the tubing operation preserves this parity.  On each side of $\Sigma'$, there is a natural pairing of $\Sigma'\cap K$ given by the collection of trivial arcs (each puncture is connected to exactly one other puncture).  However, since $c$ bounds a compressing disk for $\Sigma'_K$, all punctures of $G$ must be paired with punctures in $G$, and since $c'$ bounds a compressing disk, all punctures of $G'$ must be paired with punctures in $G'$.  It follows that there is a pairing on the punctures in $G''$, a contradiction to the fact that $G''$ meets $K$ in an odd number of points. Thus, $\Sigma'$ must be c-weakly reducible.
\end{proof}
\end{lemma}

We will need a result regarding c-essential surfaces in trivial tangles.

\begin{lemma}\label{trivial}
Suppose $B$ is $S^2 \X I$ or a 3-ball, $(B,\tau)$ is a trivial tangle, and $S$ is a properly embedded essential surface in $B \setminus \eta(\tau)$.  Then $S \cap \pd B \neq \emp$.
\begin{proof}
If $B$ is a 3-ball, then the lemma follows from Lemma 2.9 of \cite{tomova2}.  To form a contradiction, suppose that $S$ is a properly embedded essential surface in $B \setminus \eta(\tau)$ such that $S \cap \pd B = \emp$ and $B=S^2 \X I$.  Let $\{\Delta_i\}$ be a collection of bridge disks for the boundary parallel arcs of $\tau$.

By irreducibility of $B \setminus \eta(\tau)$, we can assume that there are no closed curves of intersection between $S$ and $\bigcup_i \Delta_i$. Since $S \cap \pd B = \emp$, any arc of $S\cap \bigcup_i \Delta_i$ has both boundary components in $K$. Let $S'$ be the component of $S$ that meets some $\Delta_j$ in an outermost arc $\beta$ in $\Delta_j$. The arc $\beta$ together with a subarc of $K$ cobound a disk in $\Delta_j$ that is disjoint from $S$ in its interior. The boundary of a regular neighborhood of this disk contains a compressing disk for $S$, unless $S'$ is a boundary parallel annulus. In each case, we contradict the fact that $S$ is essential. Hence, we can assume that $S$ is disjoint from $\bigcup_i \Delta_i$.

Since $S$ is disjoint from $\bigcup_i \Delta_i$, then $S$ is a properly embedded essential surface in $B \setminus \eta(\tau')$ where $\tau'$ is the vertical arcs in $\tau$. By Lemma 2.10 of \cite{tomova2}, $S$ is inessential, a contradiction.
\end{proof}
\end{lemma}

To prove the first main theorem, we need to slightly adapt the notion of width.  For an embedding $k$ of a knot $K$, define the \emph{bridge-width} $bw(k)$ to be $bw(k) = (b(k),w(k))$, and let the \emph{bridge-width} $bw(K)$ of $K$ be defined as
\[ bw(K) = \min_{k \sim K} bw(k),\]
where this minimum is taken using the dictionary ordering.  Any $k$ satisfying $bw(k) = bw(K)$ is called a \emph{bridge-thin} position, and it is clear from the definition that in this case $b(k) = b(K)$.  We note that the following theorem is implicit in the arguments of Wu in \cite{wu}:

\begin{theorem}\cite{wu}\label{thinner1}
Suppose $k$ is a bridge-thin position of $K$ which has a thin level.  Then, any thinnest thin level $h^{-1}(r_i)$ is an essential meridional surface in $E(K)$.
\begin{proof}
In \cite{wu}, Wu shows that if a thinnest thin level of an embedding $k$ of $K$ is compressible, then we may exhibit an embedding $k'$ of $K$ such that $b(k') \leq b(k)$ and $w(k') < w(k)$.  The desired statement easily follows.
\end{proof}
\end{theorem}

Finally, we may prove the main theorem from this section.

\begin{proof}[Proof of Theorem \ref{main2}]
First, we claim that $K$ is prime.  If not, there is an essential meridional annulus $A$ contained in $E(K)$.  Let $S = \pd B_1 = \pd B_2$ and choose $A$ so that $|A \cap S_K|$ is minimal.  If $S_K$ is c-compressible, then by Theorem \ref{tangdist2}, with the c-disk playing the role of the essential surface $F$, $d(\Sigma_1) \leq 2$ or $d(\Sigma_2) \leq 2$, a contradiction (note that when we apply Theorem \ref{tangdist2} $M=N=B_1$ or $M=N=B_2$). Hence, we can assume that $S_K$ is c-essential.

If $A \cap S_K \neq \emp$, then a curve $\gamma \subset A \cap S_K$ which is essential in $S_K$ gives rise to a cut disk for $S_K$, so $\gamma$ must be boundary parallel in $S_K$.  Choose such a $\gamma$ which is outermost, so that $\gamma$ bounds a once-punctured disk $D$ which avoids $A$.  However, now we may surger $A$ along $D$ to get a new essential annulus $A'$ which intersects $S_K$ fewer times.  It follows that $A \cap S_K = \emp$, so $A \subset B_1$ or $A \subset B_2$, which contradicts Lemma \ref{prime}; hence $K$ is prime.

Observe that there is an embedding $k'$ of $K$ with two thick levels parallel to $\Sigma_1$ and $\Sigma_2$ and one thin level parallel to $S = \pd B_1 = \pd B_2$, where $k'$ has bridge number $\n_1 + \n_2 - n$.  Suppose $F$ is any c-essential surface in $E(K)$. By Lemma \ref{essdist} applied to $(B_1,\tau_1)$, either $-\chi(F_K) > 2(\n_1 + \n_2 - n) - 2$ or $F \cap B_1 = \emp$ after isotopy.  In the second case, we may apply Lemma \ref{essdist} to $(B_2,\tau_2)$ to conclude that again $-\chi(F_K) > 2(\n_1 + \n_2 - n) - 2$, or else $F_K$ is isotopic to $S_K$.

Let $k$ be a bridge-thin position of $K$.  We note that if $k$ has a thick level of width at least $2(\n_1 + \n_2 - n)$, then $k$ has at least $\n_1 + \n_2 - n$ critical points and the theorem holds.  If $k$ is a bridge position with bridge sphere $\Sigma$, then by Lemma \ref{tangdist3} either $\Sigma \cap S = \emp$, $\Sigma$ is c-weakly reducible, or $-\chi(\Sigma) > 2(\n + \n' - n) - 2$.  However, the first case contradicts Lemma \ref{trivial}, and the second case implies that $k$ is not bridge-thin by Lemma \ref{cthin}.  In the third case, $k$ has a thick level of width greater than $2(\n_1+\n_2-n)$.

Suppose now that $k$ has a thin level.  By Theorem \ref{thinner1}, a thinnest thin level $R$ of $k$ is essential. Suppose that $R$ is cut-compressible. Let $\{D_i\}_{i=1}^{m}$ be a maximal sequence of cut-compressing disks taking $R$ to a cut-incompressible surface $R'$. Since the properties of being incompressible and non-boundary parallel are preserved under cut-compressing, all components of $R'$ are c-essential.  Cut-compressing $R$ along the $m-1$ cut disks $\{D_i\}_{i=1}^{m-1}$ yields a surface of some number of components, one of which, call it $R''$, contains $\pd D_m$. Let $R^{*}$ be the result of cut-compressing $R''$ along $D_m$. Since all essential curves in planar surfaces are separating, $R^*$ is a c-essential, planar, meridional surface of two components.

Since the components of $R^*$ are c-essential, either $-\chi(R_K) \geq -\chi(R^*_K) > 2(\n_1+\n_2 - n) - 2$, or each component of $R^*_K$ is isotopic to $S_K$.  In the first case, $k$ has a thick level of width greater than $2(\n_1 + \n_2 - n)$.  In the second case, since $R^{*}$ is obtained by cut-compressing $R''$, then $R''$ is obtained by tubing one component of $R^{*}$ to the other component of $R^{*}$ along a subarc of $K$. However, tubing together parallel surfaces results in a compressible surface. This contradicts the incompressibility of $R''$. Thus, we conclude that $R$ must be c-essential and as such $-\chi(R_K) > 2(\n_1 + \n_2 -n) - 2$ (completing the proof), or $R_K$ is parallel to $S_K$.

If $R_K$ is parallel to $S_K$, then $S$ is a level thin sphere with respect to $k$, and we may construct a new embedding $k^*$ from $k$ by pushing all maxima above all minima in both $B_1$ and $B_2$.  Note that $b(k^*) = b(k)$, and $k^*$ has exactly two thick spheres $\Sigma_1' \subset B_1$ and $\Sigma_2' \subset B_2$. By Theorem \ref{tangdist1}, either $\Sigma'_i$ is isotopic to $\Sigma_i$ or $-\chi((\Sigma_i)'_K) > -\chi((\Sigma_i)_K)$.  In the second case, $b(k) > b(k')$, a contradiction.  It follows that $\Sigma'_i$ is isotopic to $\Sigma_i$ for each $i\in \{1,2\}$ and $b(k) = b(k')$, as desired.
\end{proof}

\section{A high distance tangle}\label{high-distance}

In this section we use a result of Johnson and Moriah to construct a tangle with a high-distance bridge sphere that stays high distance even after many crossing changes. Their construction uses plat presentations of knots.

Let $\mathcal{B}_n$ denote the $n$-strand braid group. Given $\alpha \in \mathcal{B}_{2k}$ we can visualize $\alpha$ as a regular projection of $2k$ arcs in the plane such that each arc has one endpoint on the line $y=0$, has the other endpoint on the line $y=1$, is contained in the region between $y=0$ and $y=1$ and increases monotonically in $y$. By connecting each pair of consecutive endpoints of $\alpha$ on the line $y=1$ by an arc and similarly connecting pairs of consecutive endpoints on the line $y=0$, we form the projection of a knot $\hat \alpha$. Such a presentation of a knot is known as a \emph{$2k$-plat}.

There is a natural bridge sphere associated to a $2k$-plat. We can view the knot as being embedded in a neighborhood of its projection in the $xy$-plane in $\mathbb{R}^3$. Let $h:\mathbb{R}^3\rightarrow \mathbb{R}$ be projection onto the y-axis. After taking the one point compactification of each $xz$-plane intersecting the knot and gluing in 3-balls above and below, $h$ extends to a Morse function on $S^3$, and the 2-sphere $h^{-1}(\frac{1}{2})$ is a bridge sphere for the $2k$-plat. We call this the \emph{induced bridge sphere}.

In \cite{JM}, Johnson and Moriah study $2k$-plats of the form depicted in Figure \ref{Fig:JohnsonMoriah}. Recall that $\mathcal{B}_2$ is isomorphic to the integers, hence, the labels of $a_{i,j}$ in Figure \ref{Fig:JohnsonMoriah} dictate which element of $\mathcal{B}_2$ lies inside the corresponding braid box. A knot with a $2k$-plat presentation of the form depicted in Figure \ref{Fig:JohnsonMoriah} will be called a \emph{highly twisted plat}. Let $n$ be the number of rows associated to a highly twisted plat. Let $\lceil x \rceil$ be the ceiling function,
which is equal to the smallest integer greater than or equal to $x$.
Johnson and Moriah showed the following:

\begin{figure}[ht]
\centering
\includegraphics[scale=0.6]{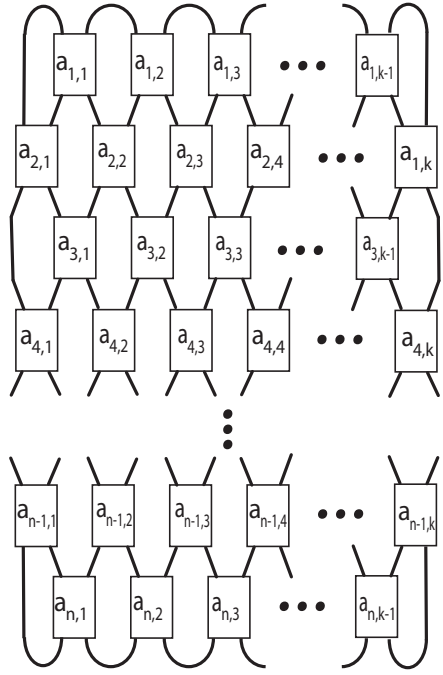}
\caption{The $2k$-plats studied by Johnson and Moriah}
\label{Fig:JohnsonMoriah}
\end{figure}

\begin{theorem}\label{thm:JohnsonMoriah}\cite{JM}
If $K \subset S^3$ is a highly twisted $2k$-plat with $k\geq 3$ and $|\, a_{i,j} |\,
\geq 3$ for all $i, j$, then $d(\Sigma) =  \lceil n /(2(k - 2))  \rceil$ where $\Sigma$ is the induced bridge sphere.
\end{theorem}

We form the tangle $(B_3, \tau_3)$ depicted in Figure \ref{Fig:highdistbraid} in the following way: Let $L$ be a highly twisted $10$-plat. Considering the plat in $\R^3$ for the moment, let $\gamma_1$ and $\gamma_2$ be two arcs which are contained in the $xy$-plane at height $y=1$, such that $\gamma_1$ has endpoints in the the second and third maxima of $L$ and $\gamma_2$ has endpoints in the fourth and fifth maxima of $L$. Then $B_3=S^3\setminus(\eta(\gamma_1)\cup \eta(\gamma_2))$ and $\tau_3=L\cap B_3$. The induced bridge sphere $\Sigma$ for $L$ persists as a bridge sphere for $(B_3, \tau_3)$. Recall that the distance of $\Sigma$ as a bridge sphere for $L$ is the distance in $\mathcal{C}(\Sigma_L)$ between two disk sets $D_1$ and $D_2$. After drilling out $\gamma_1$ and $\gamma_2$ the distance of $\Sigma$ as a bridge sphere for $(B_3, \tau_3)$ is the distance between disk sets $D'_1$ and $D'_2$ where $D'_1\subset D_1$ and $D'_2 = D_2$. Hence, the distance of $\Sigma$ as a bridge surface of $(B_3, \tau_3)$ is greater than or equal to the distance of $\Sigma$ as a bridge surface for $L$. The following result immediately follows from Theorem \ref{thm:JohnsonMoriah}.

\begin{corollary}\label{B3distbound}
Let $\Sigma$ be the induced bridge sphere for $(B_3, \tau_3)$ and suppose $|\, a_{i,j} |\,
\geq 3$ for all $i, j$.  Then $d(\Sigma) \geq  \frac{n}{6}$.
\end{corollary}

\begin{figure}[ht]
\centering
\includegraphics[scale=0.6]{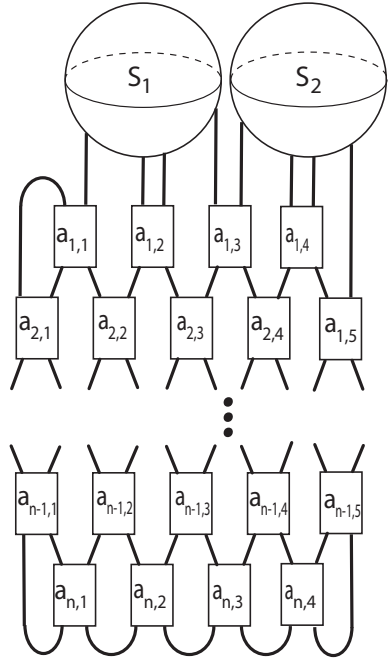}
\caption{The tangle $(B_3,\tau_3)$.}
\label{Fig:highdistbraid}
\end{figure}

\section{A knot with a compressible thin level}\label{Sec:DefinitionK}

We construct a knot $K$ in $S^3$ with a compressible thin level from $(B_3, \tau_3)$ in the following way. Let $S_1$ and $S_2$ be the two 2-spheres that make up $\partial B_3$. Let $(B_1,\tau_1)$ and $(B_2,\tau_2)$ be two 2-strand tangles with 3-bridge spheres $\Sigma_1$ and $\Sigma_2$. Glue $B_1$ to $B_3$ by identifying $\partial B_1$ with $S_1$ and glue $B_2$ to the result by identifying $\partial B_2$ with $S_2$, so that the result is a knot $K$ in $S^3$. Additionally, choose $\tau_1$ and $\tau_2$ such that $d(\Sigma_1)>12$  and $d(\Sigma_2)>12$. Similarly, choose $\tau_3$ such that $|\, a_{i,j} |\,
\geq 3$ for all $i, j$ and $n>72$. By Corollary \ref{B3distbound}, these restrictions imply that $d(\Sigma_3)>12$ where $\Sigma_3$ is the induced bridge sphere for $(B_3, \tau_3)$. Later in the paper we will need to put additional restrictions on the values of the $a_{i,j}$'s.

In the remainder of this section, we will establish several topological properties of $E(K)$ which will allow us to find thin position of $K$ in Sections \ref{Sec:thinposition} and \ref{final-case}.  From this point forward, we will suppress the notation $(S_1)_K$ and $(S_2)_K$ and consider $S_1$ and $S_2$ to be 4-punctured spheres.

\begin{lemma}\label{L1}
The 4-punctured spheres $S_1$ and $S_2$ are c-essential in $E(K)$.
\begin{proof}
Suppose by way of contradiction that $S_i$ is c-compressible for $i=1$ or 2, and choose a c-disk $D$ such that $|D \cap (S_1 \cup S_2)|$ is minimal.  If $\text{int}(D) \cap (S_1 \cup S_2) \neq \emp$, let $\gamma$ be a curve of intersection which is innermost in $D$, so that $\gamma$ bounds a c-disk $D' \subset D$ which avoids $S_1 \cup S_2$.  If $\gamma$ is essential in $S_i$, then $|D' \cap (S_1 \cup S_2)| < |D \cap (S_1 \cup S_2)|$, a contradiction.  If $\gamma$ is inessential in $S_i$, then either we can remove $\gamma$ with an isotopy of $D$, which contradicts the minimality of $|D \cap (S_1 \cup S_2)|$, or $\gamma$ is boundary parallel in $S_i$ and surgery of $D$ along $\gamma$ yields an essential annulus $A \subset B_i$, which contradicts Lemma \ref{prime}.

Thus, suppose that $\text{int}(D) \cap (S_1 \cup S_2) \neq \emp$.  Then $D$ is an essential surface in $(B_i,\tau_i)$ and $\partial D$ is essential in $S_i$, so by Theorem \ref{tangdist2}, we have $d(\Sigma_i) \leq 2$, a contradiction.
\end{proof}
\end{lemma}

\begin{lemma}\label{L1A}
The knot $K$ is prime.
\begin{proof}
As in the proof of Theorem \ref{main2}, we suppose by way of contradiction that $E(K)$ contains an essential meridional annulus $A$, chosen to intersect $S_1$ and $S_2$ minimally.  First, suppose $A \cap (S_1 \cup S_2) \neq \emp$.  Any curve of intersection $\gamma$ which is essential in $S_i$ gives rise to a cut disk, so we may assume that such curves are boundary parallel in $S_i$.  Choosing such a $\gamma$ which is outermost in $S_i$ yields a once-punctured disk $D$ which avoids $A$, and performing surgery on $A$ along $D$ yields a new essential annulus $A'$ such that $|A' \cap (S_1 \cup S_2)| < |A \cap (S_1 \cup S_2)|$, a contradiction.

On the other hand, if $A \cap (S_1 \cup S_2) = \emp$, then $A \subset B_i$ for $i=1$, $2$, or $3$, which contradicts Lemma \ref{prime}.
\end{proof}
\end{lemma}

\begin{lemma}\label{L2}
Suppose that $F$ is a c-essential planar meridional surface in $E(K)$.  Then either $F_K$ is isotopic to $S_1$ or $S_2$, or $-\chi(F_K) > 10$.
\begin{proof}
By Lemma \ref{L1A}, $K$ is prime, so by Lemma \ref{essdist}, we have that either $-\chi(F_K) > 10$ or after isotopy $F \cap B_i = \emp$ for $i = 1,2$. Hence, $F$ is a c-essential planar meridional surface embedded in the interior of $B_3$. By Lemma \ref{essdist} applied to $B_3$, $-\chi(F_K) > 10$ or $F_K$ is isotopic to $S_1$ or $S_2$.
\end{proof}
\end{lemma}

\begin{lemma}\label{L3}
Suppose that $F$ is an essential meridional 6-punctured sphere in $E(K)$.  Then $F$ is isotopic to $S_1$ tubed to $S_2$ along an arc of $\tau_3$.
\begin{proof}
By Lemma \ref{L2}, such an $F$ must be cut-compressible. Since $K$ is prime, cut-compressing $F$ yields two 4-punctured spheres $F_1$ and $F_2$. As incompressibility is preserved under cut-compressions and $K$ is prime, both $F_1$ and $F_2$ are c-essential. Hence, $F$ may be constructed by tubing $F_1$ to $F_2$ along an arc in $K$.  If $F_1$ is isotopic to $F_2$, then $F$ is compressible, and so $F_1$ and $F_2$ are non-isotopic.  By Lemma \ref{L2} (possibly after relabeling), $F_1$ is isotopic to $S_1$ and $F_2$ is isotopic to $S_2$, as desired.
\end{proof}
\end{lemma}

\section{Thin position of $K$}\label{Sec:thinposition}
First, we will describe our candidate $k'$ for thin position of $K$, and then we will eliminate all other possibilities via an exhaustive argument, proving that $k'$ is width minimizing.  The position $k'$ is depicted in Figure \ref{Fig:thinK}.  By construction $k'$ has exactly three maxima and exactly one minimum in each of $B_1$ and $B_2$, with all of the maxima above the minimum in each of these 3-balls. Similarly, $k'$ has exactly one maximum and exactly five minima in $B_3$, with the maximum above all of the minima.  Thus, $k'$ has three thick spheres, $\Sigma_1$, $\Sigma_3$, and a third which we will call $\Sigma_2'$.  Note that $\Sigma_2'$ intersects $B_2$ in a surface which is isotopic to $\Sigma_2$ after a single compression.  The two thin spheres of $k'$ are $S_1$ and an 8-punctured sphere we will denote by $S_3$. Note that after a single compression $S_3$ is isotopic to $S_1\cup S_2$. The thin-thick tuple for $k'$ is $(10,8,10,4,6)$, which gives $w(k') = 78$ by equation (\ref{alt}).

\begin{figure}[ht]
\centering
\includegraphics[scale=0.6]{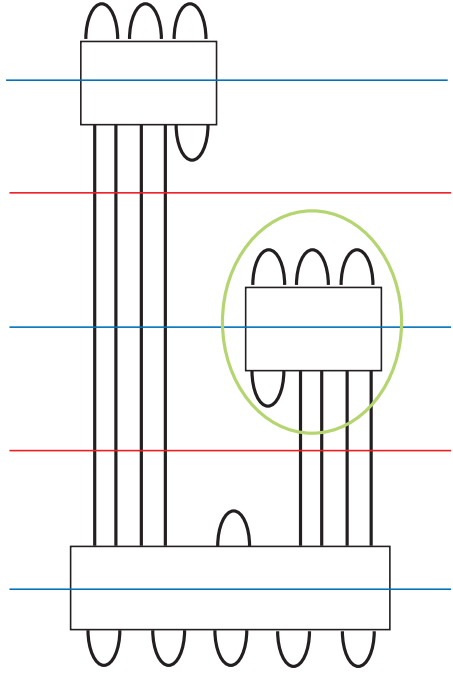}
\caption{This embedding is $k'$, the candidate thin position of $K$.}
\label{Fig:thinK}
\end{figure}

Let $k$ denote any other embedding of $K$.  We will use a case-by-case analysis to show $w(k) \geq 78$.

\begin{case}\label{C1}
The embedding $k$ is a bridge position.
\begin{proof}
Suppose that $k$ is a bridge position with bridge sphere $\Sigma$.  By Lemma \ref{tangdist3}, either $\Sigma$ can be made disjoint from $S_1$ via an isotopy transverse to $K$, $\Sigma$ is c-weakly reducible, or $-\chi(\Sigma_K) > 10$.  In the first case, $S_1$ is an essential surface for one of the trivial tangles bounded by $\Sigma$, a contradiction to Lemma \ref{trivial}.  In the second case, $k$ is not thin by Lemma \ref{cthin}.  Finally, if $-\chi(\Sigma_K) > 10$, then $w(\Sigma) \geq 14$ and thus $w(k) \geq 98$ by inequality (\ref{ineq1}).
\end{proof}
\end{case}

It follows that $k$ must have a thin level.

\begin{case}\label{C3}
The embedding $k$ has exactly one thin level.
\begin{proof}
Let $S$ be the unique thin level for $k$. By Theorem \ref{thinnest}, we can assume that $S$ is c-essential. By Lemma \ref{L2}, either $S_K$ is isotopic to $S_1$ or $S_2$, or $-\chi(S_K) > 10$.  In the latter case, $k$ has a thick level with at least 16 punctures, and using inequality (\ref{ineq1}) we get $w(k) \geq 128$. Hence, we can assume that $S_K$ is isotopic to $S_1$ or $S_2$.

Without loss of generality, suppose that $S_K$ is isotopic to $S_1$ and that $B_1$ is above $S_1$. Let $\Sigma_1'$ denote the thick level above $S_1$. We could easily show that $\Sigma_1' = \Sigma_1$, but for this case we need only that $w(\Sigma_1') \geq 6$, which follows immediately from the fact that $K$ meets $B_1$ in a non-trivial tangle.  Let $\Sigma$ denote the thick level below $S_1$.  It follows that $\Sigma$ is a bridge sphere for $(S^3 \setminus \eta(B_1), K \setminus \eta(\tau_1))$, and so by Theorem \ref{tangdist1}, we have that $-\chi(\Sigma_K) \geq d(\Sigma_3) - 2 > 10$ or else $\Sigma$ is isotopic to $\Sigma_3$ after some number of c-compressions of $\Sigma \cap B_3$.  In the first case, $w(\Sigma) \geq 14$ and $w(k) \geq 98$. In the second case, if no c-compressions are necessary, then $\Sigma$ is isotopic to $\Sigma_3$. In this case, $S_2$ is an essential surface embedded in the trivial tangle cobounded by $\Sigma$ and $S_1$, this contradicts Lemma \ref{trivial}. If it is necessary to c-compress $\Sigma$ to produce a surface isotopic to $\Sigma_3$, then $w(\Sigma) > w(\Sigma_3) = 10$, so $w(\Sigma) \geq 12$, and by equation (\ref{alt}),
\[ w(k) \geq \frac{1}{2}\left(6^2 + 12^2 - 4^2\right) = 82.\]
\end{proof}
\end{case}

\begin{case}\label{C4}
The embedding $k$ has two or more thin levels, and one has width at least 8.
\begin{proof}
If $S$ is a thin level of $k$ and $w(S) = 8$, then $k$ has two thick levels adjacent to $S$ with width at least 10. Since $K$ is prime, any other thin level must have width at least 4 and any thick level not adjacent to $S$ must have width at least 6. Hence, we have
\[ w(k) \geq  \frac{1}{2}\left(6^2 + 10^2 +10^2 - 4^2 - 8^2\right) = 78.\]
\end{proof}
\end{case}

\begin{case}\label{C5}
The embedding $k$ has exactly two thin levels, both of width 4.
\begin{proof}
As in Case \ref{C3}, we may assume that one of the thin levels is isotopic to $S_1$.  If $S$ is the other thin level with width 4, then $S$ is a thinnest thin level since $K$ is prime. By Theorem \ref{thinner}, $S$ is essential and, since $K$ is prime, must also be c-essential. By Lemma \ref{L2}, $S$ must be isotopic to $S_1$ or $S_2$, or $k$ is not thin.  But if $S$ is isotopic to $S_1$, then by the isotopy extension theorem we may replace the embedded arcs of $k$ between $S_1$ and $S$ with vertical arcs, reducing $w(k)$.  Thus, $S$ is isotopic to $S_2$.  Let $\Sigma$ denote the thick level between $S_1$ and $S_2$, and note that the region between $S_1$ and $S_2$ is $B_3$.

By Theorem \ref{tangdist1}, either $-\chi(\Sigma_K) \geq d(\Sigma_3) - 2 > 10$, or after isotopy, $\Sigma$ c-compresses to a surface isotopic to $\Sigma_3$.  In the first case $w(\Sigma) \geq 14$, and in the second case, we note that $\Sigma$ separates $S_1$ and $S_2$, while $\Sigma_3$ does not, so we must c-compress $\Sigma$ at least once before it is isotopic to $\Sigma_3$, implying that $w(\Sigma) > w(\Sigma_3) = 10$.  In either case, $w(\Sigma) \geq 12$, and by equation (\ref{alt}),
\[ w(k) \geq \frac{1}{2}\left(6^2 + 12^2 +6^2 - 4^2 - 4^2\right) = 92.\]
\end{proof}
\end{case}

\begin{case}\label{C6}
The embedding $k$ has three or more thin levels.
\begin{proof}Note that since $K$ is prime, all thin levels of $k$ have width at least 4 or $k$ is not thin.
If at least one thin level of $k$ has width at least 8, then

\[ w(k) \geq \frac{1}{2}\left(6^2 + 6^2 +10^2 +10^2 - 4^2 - 4^2 - 8^2\right) = 88.\]

Hence, we can assume that all thin levels for $k$ have width 4 or 6.

First, consider the case when $k$ has at least two thin levels of width 6. Call these thin levels $S$ and $S'$. By Theorems \ref{thinner} and \ref{plustwo}, $S$ and $S'$ are incompressible, which implies that both $S$ and $S'$ are constructed by tubing $S_1$ to $S_2$ along some arc of $\tau_3$, by Lemma \ref{L3}. Since $S$ and $S'$ are disjoint, they are constructed by tubing $S_1$ to $S_2$ along a common arc of $\tau_3$. Hence, $S$ is isotopic to $S'$. By the isotopy extension theorem we may replace the embedded arcs of $k$ between $S$ and $S'$ with vertical arcs, reducing $w(k)$.

Next, consider the remaining cases when $k$ has two or more thin levels of width 4. Call two such thin levels $S$ and $S'$. As argued in Case \ref{C5}, each of $S$ and $S'$ are isotopic to $S_1$ or $S_2$. Without loss of generality, suppose $S$ is isotopic to $S_1$. If $S'$ is isotopic to $S_1$, then by the isotopy extension theorem we may replace the embedded arcs of $k$ between $S$ and $S'$ with vertical arcs, reducing $w(k)$.  Thus, $S'$ is isotopic to $S_2$.  Let $S''$ be a third thin level of $k$. Since $4\leq w(S'') \leq 6$, and $S''$ cannot be isotopic to $S_1$ or $S_2$, then $w(S'') = 6$.  By Theorem \ref{plustwo}, $S'$ is incompressible, which implies that $S''$ may be constructed by tubing $S_1$ to $S_2$ by Lemma \ref{L3}.

However, any three level spheres cut $S^3$ into four components (two components that are homeomorphic to $B^3$ and two components that are homeomorphic to $S^2\times I$), where any component is adjacent to one or two other components.  On the other hand, $S^3 \setminus \eta(S_1 \cup S_2 \cup S'')$ has four components, one of which (the component containing the tube) is adjacent to all three other components, a contradiction.
\end{proof}
\end{case}

In summary, assuming $k$ is thin, we have ruled out Cases \ref{C1}, \ref{C3} and \ref{C6}; hence we can assume that $k$ has exactly two thin levels. We have also ruled out Cases \ref{C4} and \ref{C5}, so we can assume that one thin level has width 4 and the other has width 6. As previously argued, any thin level of width 4 must be isotopic to $S_1$ or $S_2$. Without loss of generality, we will assume that $S_1$ is a thin level. This one remaining case (the case in which $k$ has two thin levels, one isotopic to $S_1$ and the other having width 6) is more complicated than the others, and we examine it in the next section.

\section{The final remaining case}\label{final-case}

To complete the proof of the main theorem, we suppose that $k$ has two thin levels, one of which is $S_1$, and the other of which we denote by $S^*$. By Theorem \ref{plustwo}, $S^*$ is incompressible, which implies that $S^*$ may be constructed by tubing $S_1$ to $S_2$ along an arc of $\tau_3$, by Lemma \ref{L3}.

By Corollary \ref{B3distbound}, if $(B_3, \tau_3)$ is altered in a way that increases $|a_{ij}|$ for any $i$ and $j$, then the distance bound on $\Sigma_3$ is preserved and all of the properties of $K$ and $k$ that were established in Section \ref{Sec:DefinitionK} and Section \ref{Sec:thinposition} are preserved. Denote the strands of $\tau_3$ by $t_1$, $t_2$, $t_3$ and $t_4$. To motivate our choices of $a_{ij}$, we need a clear picture of the surface $S^*$.

Suppose for the moment that $S^*$ is $S_1$ tubed to $S_2$ along $t_1$, the leftmost arc in Figure \ref{tangler2}.  Let $B^*$ denote that ball bounded by $S^*$ which is disjoint from $S_1 \cup S_2$, and let $\tau^* = K \cap B^*$.  Then $(B^*,\tau^*)$ is a 3-strand tangle.  Now, we may construct an isotopy of $S^*$ which drags $S_1$ from one endpoint of $t_1$ to the other.  An appropriate analogy is to picture a cord being pulled back into a vacuum cleaner, where in this setting $S_2$ takes the place of the vacuum cleaner.  As such, each of the three arcs of $(B^*,\tau^*)$ could be constructed by taking one of $t_2$, $t_3$, or $t_4$ and attaching its endpoint to $t_1$ in $S_1$.  See Figure \ref{tangler2}.  We wish to show that, under certain assumptions on the parameters $a_{ij}$, each of these arcs is knotted in $B^*$.

\begin{figure}[h!]
  \centering
    \includegraphics[width=.7\textwidth]{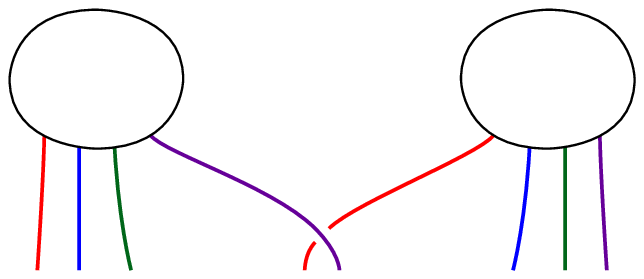}
    \includegraphics[width=.7\textwidth]{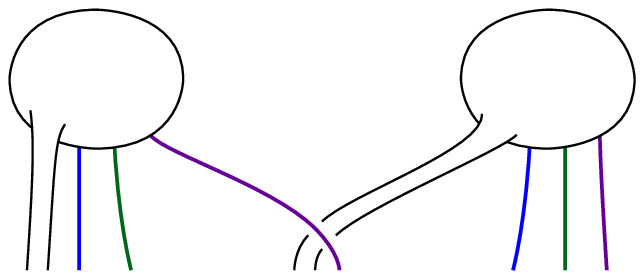}
    \includegraphics[width=.7\textwidth]{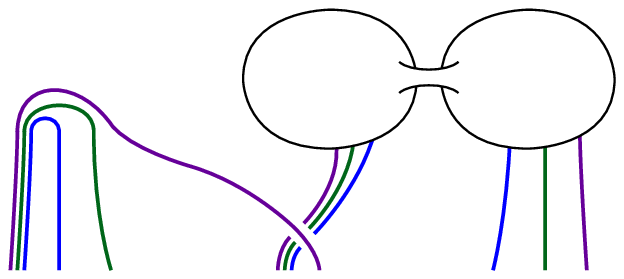}
    \caption{At top, a portion of the knot $K$, shown with $S_1$ and $S_2$.  In the middle, tubing $S_1$ to $S_2$ along $a$ yields $S^*$.  At bottom, an alternate picture of $S^*$ created by dragging $S_1$ along $a$ until it meets $S_2$.}
    \label{tangler2}
\end{figure}

We determine the parameters $a_{ij}$ in the following way: $a_{1,1}$, $a_{1,3}$, $a_{2,3}$, $a_{2,4}$, $a_{3,2}$, $a_{3,3}$ and $a_{3,4}$ are odd positive integers greater than or equal to 3, $a_{1,2}$, $a_{1,4}$, $a_{2,1}$, $a_{2,2}$, $a_{2,5}$ and $a_{3,1}$ are even integers with absolute value greater than 3. Moreover, we require $a_{2,1}$ and $a_{2,2}$ to be negative and $a_{3,1}$ to be positive. If $i$ is even and greater than 3, then $a_{i,j}$ is a negative even integer less than -3 for all $j$. If $i$ is odd and greater than 4, then $a_{i,j}$ is a positive even integer greater than or equal to 3 for all $j$. These restrictions are depicted in Figure \ref{tanglechoice}. In that figure, a label of $E$ indicates that $a_{i,j}$ is even and a label of $O$ indicates $a_{i,j}$ is odd. Similarly, the superscripts in the figure indicate the sign of $a_{i,j}$. The absence of a superscript indicates that choosing a sign is unnecessary for our construction.

\begin{figure}[h!]
  \centering
    \includegraphics[scale=0.8]{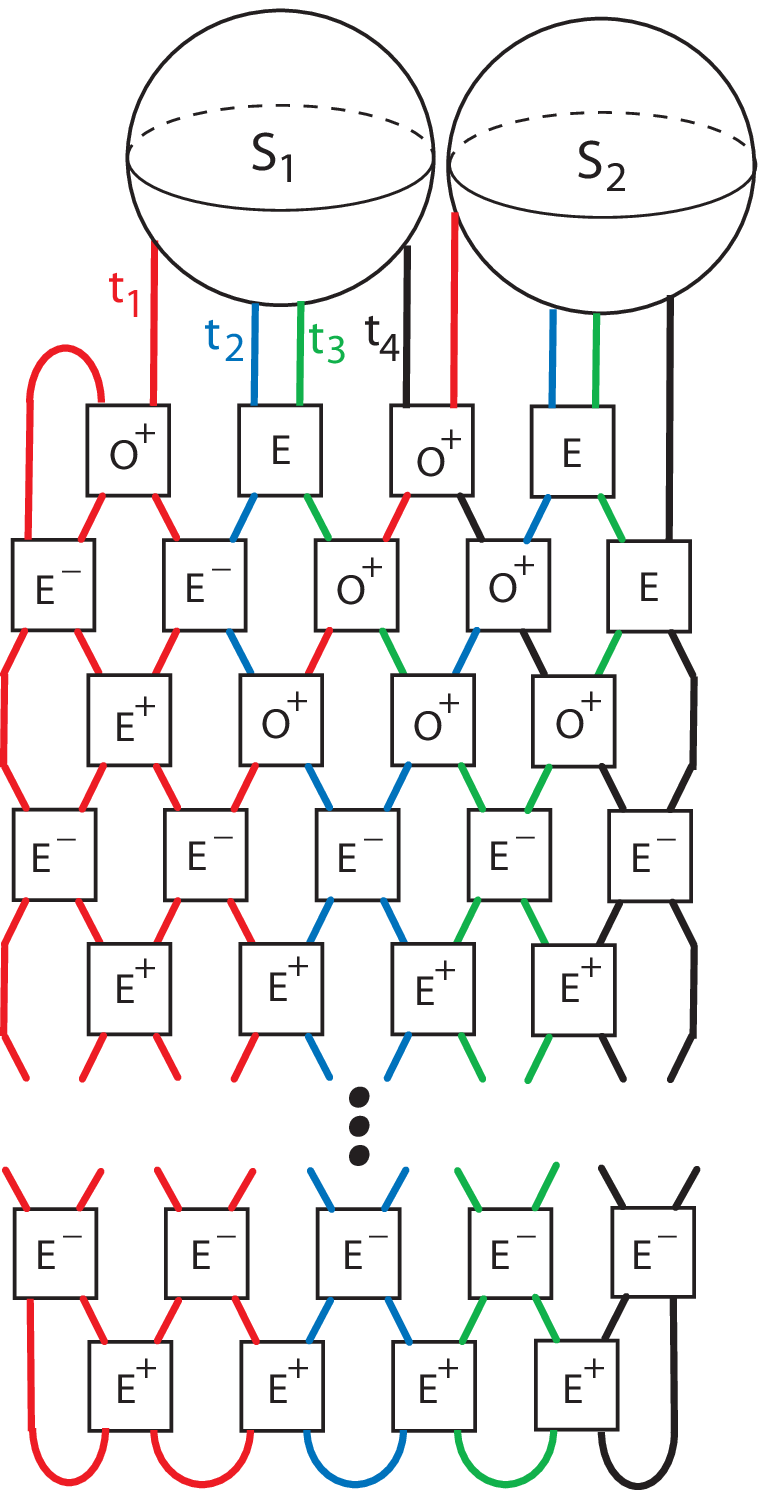}
    \caption{Restrictions on the $a_{i,j}$ for the tangle in $B_3$.}
    \label{tanglechoice}
\end{figure}

Note that there are unique arcs $s_{12}^1$ and $s_{12}^2$ connecting the endpoints of $t_1$ and $t_2$ in $S_1$ and $S_2$ when we ignore the endpoints of the strands of $\tau_3$ that are not $t_1$ or $t_2$. Let $K_{12}$ denote the knot $t_1 \cup s_{12}^1\cup t_2 \cup s_{12}^2$. We can similarly define the knots $K_{ij}$ where $i<j$. Using Figure \ref{tanglechoice}, we can easily obtain diagrams for each of these knots. These diagrams are depicted in Figure \ref{Ks}. Note that each of these diagrams is reduced and alternating. Hence, each of the knots $K_{12}$, $K_{13}$, $K_{14}$, $K_{23}$, $K_{24}$ and $K_{34}$ is nontrivial.

\begin{figure}[h!]
  \centering
    \includegraphics[width=.7\textwidth]{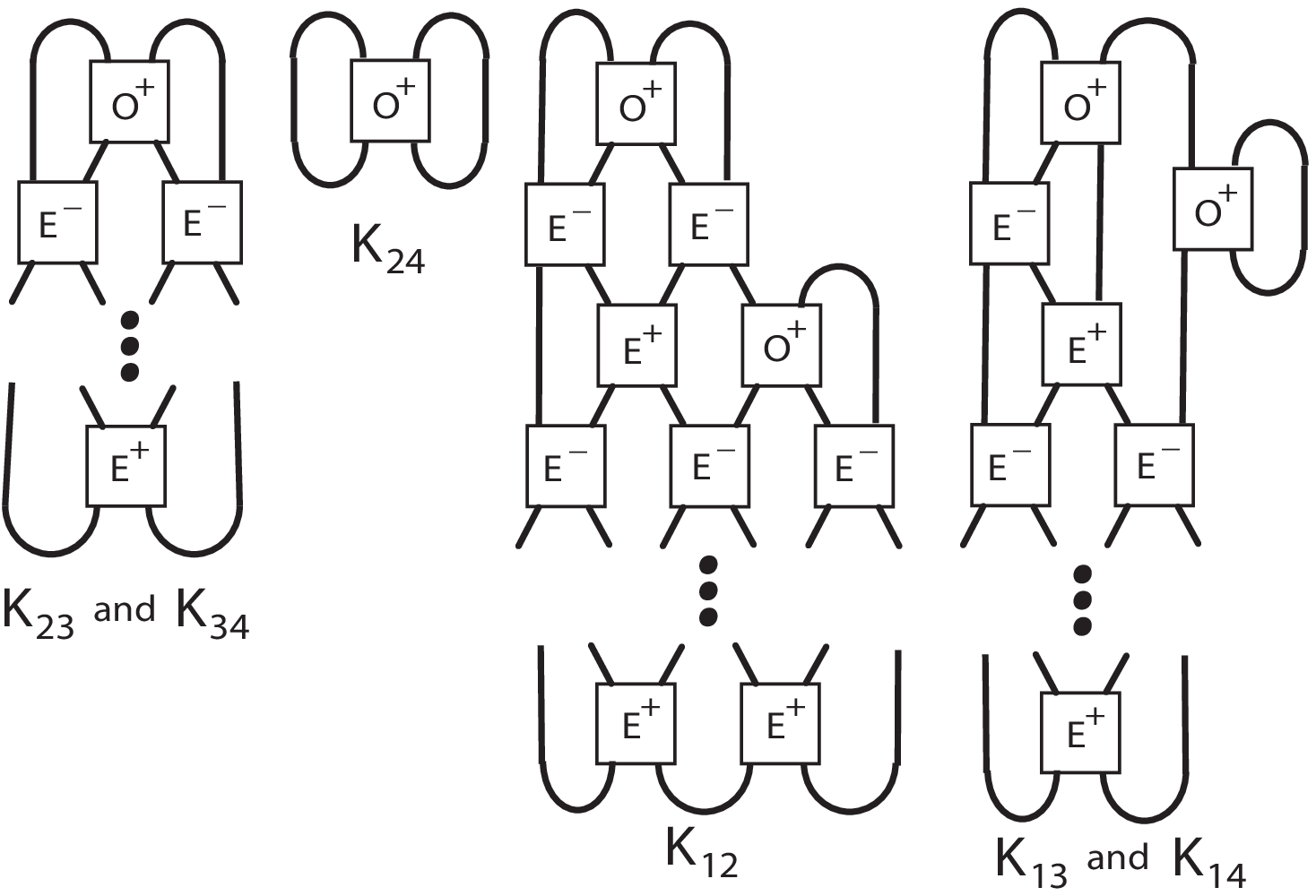}
    \caption{Diagrams for each of the knots $K_{12}$, $K_{13}$, $K_{14}$, $K_{23}$, $K_{24}$ and $K_{34}$.}
    \label{Ks}
\end{figure}

\begin{case}\label{C7}
The embedding $k$ has two thin levels, one of width 4 and the other of width 6.
\begin{proof}
As above, we may assume without loss of generality that one of the thin levels is $S_1$ and the other, $S^*$, is the result of tubing $S_1$ to $S_2$ along some strand $t_i$ of $\tau_3$.  Note that $S^*$ bounds two 3-balls in $S^3$; let $B^*$ denote the one which is disjoint from $B_1 \cup B_2$.  Suppose without loss of generality that $S_1$ is above $S^*$ with respect to the height function $h$. Hence, $k$ has a thick level $\Sigma^*$ below $S^*$, which implies that $\Sigma^* \subset B^*$.

Let $\tau^* = K \cap B^*$, and consider the 3-strand tangle $(B^*,\tau^*)$, noting that $\Sigma^*$ is a bridge surface for $(B^*,\tau^*)$.  We will show that each arc $t$ of $\tau^*$ intersects $\Sigma^*$ at least four times.  If not, then $(B^*,t)$ is a trivial 1-strand tangle.  However, we may view $t$ as the union of $t_i$ and $t_j$ with the unique arc $s^{2}_{ij}$ connecting $\pd t_i$ to $\pd t_j$ in $S_2$.  If $|t \cap \Sigma^*| = 2$, then (ignoring the other strands of $\tau^*$) $t$ is isotopic into $S^*$, and it follows that one of the knots $K_{12}$, $K_{13}$, $K_{14}$, $K_{23}$, $K_{24}$ or $K_{34}$ is the unknot, a contradiction.

We conclude that $w(\Sigma^*) \geq 12$.  Since $k$ has two other thick levels of widths at least 6 and 8, we have
\[ w(k) \geq \frac{1}{2}\left(6^2 + 8^2 +12^2 - 4^2 - 6^2\right) = 96.\]
\end{proof}
\end{case}

In summary, we have the following:

\begin{proof}[Proof of Theorem \ref{main}]
Let $k'$ be the embedding of $K$ depicted in Figure \ref{Fig:thinK}.  Noting that $S_3$ is a compressible thin level of $k'$, we need only show that $w(K) = 78$, so that $k'$ is thin.  Let $k$ be a thin position of $K$. By Cases \ref{C1}, \ref{C3} and \ref{C6}, we can assume that $k$ has exactly two thin levels. By Cases \ref{C4} and \ref{C5}, we can assume that one thin level has width 4 and the other has width at least 6.  If we assume that the thin levels have width exactly 4 and 6, then $w(k) \geq 96$ by Case \ref{C7}, so $k$ is not thin.  It follows that the three thick levels of $k$ have widths at least 6, 10, and 10; hence
\[ w(k) \geq \frac{1}{2}\left(6^2 + 10^2 +10^2 - 4^2 - 8^2\right) = 78.\]
\end{proof}

\end{document}